\documentclass{amsart}
\usepackage[english]{babel}
\usepackage[utf8]{inputenc}
\usepackage[autostyle]{csquotes}

\usepackage{latexsym}
\usepackage{amsfonts}
\usepackage{euscript}
\usepackage{times}

\usepackage{dsfont}
\usepackage{amsmath}
\usepackage{enumitem}
\usepackage{amsthm, amssymb}
\usepackage{bm}
\usepackage{url}

\usepackage{booktabs}
\usepackage{caption}

\usepackage{imakeidx}
\makeindex[columns=1]

\usepackage{cite}

\theoremstyle{plain} 
\newtheorem{teo}{Theorem}[section]
\theoremstyle{definition}
\newtheorem{defi}[teo]{Definition}
\theoremstyle{plain} 
\newtheorem{prop}[teo]{Proposition}
\theoremstyle{plain}

\theoremstyle{plain}
\newtheorem{cor}[teo]{Corollary}	
\theoremstyle{definition}

\theoremstyle{definition}

\theoremstyle{definition}
\newtheorem{ese}[teo]{Example}
\theoremstyle{plain}

\newcommand*{\sn}{\unlhd \unlhd \ }
\DeclareMathOperator{\spr}{sp}

\DeclareMathOperator{\dn}{dn}

\DeclareMathOperator{\fix}{Fix}

\begin{document}

\title{Subnormalizers and the degree of nilpotence in finite groups}
\author{Pietro Gheri}
\address{Dipartimento di Matematica e Informatica ``U. Dini'',\newline
Universit\`a degli Studi di Firenze, viale Morgagni 67/a,
50134 Firenze, Italy.}
\email{pietro.gheri@unifi.it}
\begin{abstract}
We present a CFSG-free proof of the fact that the degree of nilpotence of a finite nonnilpotent group is less than $1/2$. 
\end{abstract}

\maketitle

\section{Introduction and tools}

Let $G$ be a finite group. The \textit{degree of commutativity of $G$} is the probability that two randomly chosen elements of $G$ commute. As a natural generalization, the \textit{degree of nilpotence of $G$} is defined to be the probability that two randomly chosen elements of $G$ generate a nilpotent subgroup, that is
\begin{equation} \label{degree of nilpotence}
\dn(G)=\dfrac{\left| \lbrace (x,y) \in G \ | \ \langle x,y \rangle \text{ is nilpotent } \rbrace \right|}{\left| G \right|^2} = \dfrac{1}{|G|} \sum_{x \in G} \dfrac{|Nil_G(x)|}{|G|}.
\end{equation}
where for any $x \in G$, $Nil_G(x)$ is the set of elements $y \in G$ such that $\langle x,y \rangle$ is nilpotent.

In 1978 Gustafson proved that the degree of commutativity of a finite nonabelian group is less or equal than $5/8$ (\cite{gustafson:degcom}). In \cite{guralnick:solvable} Guralnick and Wilson proved an analogous theorem for the degree of nilpotence. 
\begin{teo} \label{dn12 senza classificazione} 
Let $G$ be a finite group. If $\dn(G) >1/2 $ then $G$ is nilpotent. The value $1/2$ is tight.
\end{teo}
This was obtained as a corollary of a similar result concerning the probability that two elements generate a solvable subgroup. However, while Gustafson's proof was very short and only involved basic tools, Guralnick and Wilson's proof used the Classification of finite simple groups. In that work the authors ask if there is a Classification-free proof of their result on the degree of nilpotence. In this paper we present such a proof.

We immediately clarify that the tightness of the value $1/2$ is trivial, since one can verify directly that $\dn (S_3) =1/2$.

The first step in our proof consists of replacing the set $Nil_G(x)$ in (\ref{degree of nilpotence}) with another set, that is $S_G(\langle x \rangle)$,  the \textit{Wielandt's subnormalizer of $\langle x \rangle$ in $G$}. The reasons of this substitution will be clarified in the last section.

\begin{defi}[\cite{lennox:subnormal}, page 238]\label{subnormalizer} Let $H$ be a subgroup of $G$. The \textit{subnormalizer of $H$ in $G$} is the set
\[
S_G(H) = \lbrace g \in G \ | \ H \sn \langle H,g \rangle \rbrace.
\]
\end{defi}

Wielandt's subnormality criterion can be restated using this definition: a subgroup $H$ of a finite group $G$ is subnormal if and only if $S_G(H)=G$.

A crucial result for our proof is a theorem proved by Casolo in 1990 which gives a formula to count the elements of the subnormalizer of a $p$-subgroup. Given a prime $p$ dividing the order of $G$ and a $p$-subgroup $H$ of $G$, we write $\lambda_G(H)$ for the number of Sylow $p$-subgroups of $G$ containing $H$. When $H = \langle x \rangle $ is a cyclic subgroup we write $S_G(x)$ and $\lambda_G(x)$, in place of $S_G(\langle x \rangle)$ and $\lambda(\langle x \rangle)$.
\begin{teo}[\cite{casolo:subnor}] \label{form subnor theorem}  Let $H$ be a $p$-subgroup of $G$ and $P \in Syl_p(G)$. Then
\begin{equation} \label{formulasubnor}
|S_G(H)|=\lambda_G(H) |N_G(P)|.
\end{equation}
\end{teo}
The proof of this result does not rely on CFSG.

With the degree of nilpotence in mind, we can then define a new probability where, as we said before, $Nil_G(x)$ is replaced by $S_G(x)$. If we set $\spr_G(x)=|S_G(x)|/|G|$ we have that $\spr_G(x) \geq |Nil_G(x)|/|G|$ since every subgroup of a finite nilpotent group is subnormal. Then we set
\[
\spr (G) = \frac{1}{|G|} \sum_{x \in G} \frac{|S_G(x)|}{|G|}.
\]
Moreover if $x$ is a $p$-element, using Theorem \ref{form subnor theorem} the ratio $\spr_G(x)$ can be written as
\begin{equation} \label{scrittura spr}
\spr_G(x)=\frac{|S_G(x)|}{|G|} = \frac{\lambda_G(x)|N_G(P)|}{|G|}= \frac{\lambda_G(x)}{n_p(G)}  
\end{equation}
where $n_p(G)$ is the number of Sylow $p$-subgroups of $G$. The value $\spr_G(x)$ is then the percentage of Sylow $p$-subgroup of $G$ containing $x$.

\section{Proof}

The main step for our proof of Theorem \ref{dn12 senza classificazione} is the following probabilistic version for cyclic subgroups of Wielandt's subnormality criterion.

\begin{prop} \label{uno su p piu uno}
Let $p$ be a prime dividing the order of $G$, $x \in G$ be a $p$-element of order $p^r$ and $1 \leq k \leq r$. If $\spr_G(x)>1/(p^k+1)$ then $x^{p^{k-1}} \in O_p(G)$.
\end{prop}
\begin{proof}
By \ref{formulasubnor} $\spr_G(x)= \lambda_G(x) / n_p(G)$. Let $y_1, \dots , y_{p^k+1}$ be $p^k+1$ distinct conjugates of $x$. Then there exists $P \in Syl_p(G)$ such that two of these conjugates both belong to $P$. For if not, 
\[
U_i = \lbrace P \in Syl_p(G) \ | \ y_i \in P \rbrace, \ i \in \lbrace 1, \dots, p^k+1 \rbrace,
\]
would be disjoint sets, each of cardinality $\lambda_G(x)$ (since the function $\lambda_G$ is constant on the conjugacy classes of $p$-elements) and we would have
\[
n_p(G) \geq \left| \bigcup_{i=1}^{p^k+1} U_i \right| = (p^k+1)\lambda_G(x) 
\]
against the hypothesis.

Let $g \in G$ and set $y_0=x$, $y_i=x^{gx^{i-1}}$ for $1 \leq i \leq p^k$. We then have two cases: either there exist $0 \leq i< j \leq p^k$ such that $y_i=y_j$ or the set of the $y_i$'s has cardinality $p^k+1$. In any case we then have that the following statement holds.
\begin{equation} \tag{$\ast$}\label{cond2}
\text{ There exist } 0 \leq i < j \leq p^k \text{ such that } \langle y_i,y_j \rangle \text{ is a $p$-group}
\end{equation}
We want to prove that if (\ref{cond2}) holds for all $g \in G$ then $x^{p^{k-1}} \in O_p(G)$. Arguing by induction on $|G|$ we can suppose that $x^{p^{k-1}} \in O_p(H)$, that is $ \langle x^{p^{k-1}} \rangle$ is subnormal in $H$, for all proper subgroups $H$ of $G$ containing $x^{p^{k-1}}$. By the Wielandt's zipper lemma \cite[Lemma 7.3.1]{lennox:subnormal} $x^{p^{k-1}}$ is contained in a unique maximal subgroup $M$ of $G$, which is not normal in $G$.

If $1 \leq s < p^k$ then $x^{p^{k-1}} \in \langle x^s \rangle$ and so $M$ is the unique maximal subgroup containing $x^s$. Moreover if for some $a \in G, \ (x^s)^a \in M$ then $x^s \in M^{a^{-1}}$ and so $M=M^{a^{-1}}$. Since $M$ is maximal and is not normal in $G$, we have $a \in M$.

Let then $g \in G$, $y_i$ be defined as above and suppose that (\ref{cond2}) holds. We separately consider two cases: one in which $i=0$ and the other in which $i \geq 1$. If $i=0$ then $\langle x, y_j \rangle$ is a $p$-group, which implies that $y_j =x^{gx^{j-1}}\in M$. It follows that $gx^{j-1} \in M$ and so $g \in M$. If instead $i \geq 1$ then $\langle y_i,y_j \rangle$ is a $p$-group and so is the subgroup 
\[
\langle x^{(x^{j-i})^{g^{-1}}},x\rangle= \langle y_i,y_j \rangle^{x^{-(i-1)}g^{-1}}. 
\]
It follows that $x^{(x^{j-i})^{g^{-1}}} \in M$, so $(x^{j-i})^{g^{-1}} \in M$ and finally $g^{-1} \in M$. 

We proved that if (\ref{cond2}) holds then $g \in M$. It follows that $G \leq M$, a contradiction.
\end{proof}

The bound in the previous proposition is the best possible, as we can see looking at $G=PSL(2,p)$. We have that each Sylow $p$-subgroup of $G$ has cardinality $p$, $n_p(G)=p+1$ and $O_p(G)=1$. Then if $x \in G$ is a $p$-element we have that $\spr(G)=1/(p+1)$. 

\begin{cor} \label{fitting un terzo} Let $x \in G$ be an element that does not lie in the Fitting subgroup of $G$. Then $\spr_G(x) \leq 1/3$.
\end{cor}
\begin{proof}
Let $p$ be a prime dividing the order of $x$ such that the $p$-part $x_p$ of $x$ does not lie in $O_p(G)$. Then by Proposition \ref{uno su p piu uno} we have $\spr_G(x) \leq \spr_G(x_p) \leq 1/(p+1)$.
\end{proof}

We can now prove Theorem \ref{dn12 senza classificazione}.

\begin{proof}
First of all we observe that $\left[ G: \mathbf{F}(G) \right] \leq 3$. For if $\left[ G: \mathbf{F}(G) \right] \geq 4$ then by Corollary \ref{fitting un terzo}
\begin{equation}
\begin{split}
\dn(G) & \leq \spr(G) = \frac{1}{|G|} \sum_{x \in G} \frac{|S_G(x)|}{|G|}  \\
& =\frac{1}{|G|} \left( |\mathbf{F}(G)|+ \sum_{x \notin \mathbf{F}(G)} \frac{|S_G(x)|}{|G|} \right)  \\
& \leq \frac{1}{|G|} \left( |\mathbf{F}(G)|+ \frac{1}{3} \frac{|G|-|\mathbf{F}(G)|}{|G|} \right)  \\
& = \frac{|G| + 2 |\mathbf{F}(G)|}{3|G|} \leq \frac{1}{3} + \frac{1}{6} =\frac{1}{2}.
\end{split}
\end{equation}
Thus $\left[ G: \mathbf{F}(G) \right] \in \lbrace 2,3 \rbrace$. Let $G$ be a counterexample of minimal order. It is an easy exercise to verify that setting $N:=\mathbf{F}(G)$, we have $G=N\langle x \rangle $ with $|x|=q \in \lbrace 2,3 \rbrace$ and $N$ is an elementary abelian group of order say $p^k$, for some prime $p \neq q$. Moreover $C_N(x)=1$ and $G$ is a Frobenius group with kernel $N$. For every $1 \neq a \in N$ we have
\[
Nil_G(a)=N,
\]
while for every $y \notin N$ we have
\[
Nil_G(y)=\langle y \rangle.
\]
Therefore
\begin{equation*}
\begin{split}
\dn(G)= \frac{1}{|G|} \sum_{g \in G} \frac{|Nil_G(g)|}{|G|} = \frac{1}{|G|} \left( \sum_{g \in N} \frac{|Nil_G(g)|}{|G|}+ \sum_{g \notin N} \frac{|Nil_G(g)|}{|G|} \right)  \\
= \frac{1}{p^kq} \left( 1+(p^k-1) \frac{1}{ q}+ (q-1) p^k \frac{1}{p^k} \right)= \frac{1}{p^kq}+\frac{p^k-1}{p^k}\frac{1}{q^2}+\frac{q-1}{p^kq},
\end{split}
\end{equation*}
which is greater then $1/2$ if and only if $q=2$ and $p^k=3$, that is if and only if $G \simeq S_3$. By direct calculation one see that $\dn(S_3)=1/2$ and so we have the thesis.
\end{proof}

The next theorem is another result with the flavour of Gustafson's theorem, concerning the probability $\spr(G)$.

\begin{teo} \label{sp nil}
If $\spr(G)>2/3 $ then $G$ is nilpotent, and the bound is the best possible.
\end{teo}
\begin{proof}
This is just a calculation which follows easily from Corollary \ref{fitting un terzo}. Let $G$ be a nonnilpotent group: then $|\mathbf{F}(G)| \leq |G|/2$. Thus
\begin{equation*}
\begin{split}
\spr(G) & = \frac{1}{|G|}\sum_{g \in G} \spr_G(x)= \frac{1}{|G|}\sum_{g \in \mathbf{F}(G)} \spr_G(x) + \frac{1}{|G|}\sum_{g \notin \mathbf{F}(G)} \spr_G(x) \\
& \leq \frac{|\mathbf{F}(G)|}{|G|} + \frac{1}{3} \frac{|G \setminus \mathbf{F}(G)|}{|G|} = \frac{1}{3}+ \frac{2}{3}\frac{|\mathbf{F}(G)|}{|G|} \leq \frac{2}{3}.
\end{split}
\end{equation*}
The fact that the bound is tight follows from an easy calculation that gives $\spr(S_3)=2/3$.
\end{proof}

\section{Some examples and remarks}

In this section we first of all explain why it is crucial for our proof of Theorem \ref{dn12 senza classificazione} to replace $Nil_G(x)$ with $S_G(x)$. Looking at Gustafson's proof about the degree of commutativity (\cite{gustafson:degcom}) we find that a fundamental fact is that if $x \in G$ satisfies $|C_G(x)|>|G|/2$ then $x \in Z(G)$, that is $C_G(x)=G$. Corollary \ref{fitting un terzo} gives a similar result concerning $S_G(x)$: if $|S_G(x)|>|G|/3$ then $x \in \mathbf{F}(G)$, that is $S_G(x)=G$. It is not difficult to see that the elements $x$ such that $Nil_G(x)=G$ are exactly the elements in $\zeta_\omega(G)$, the hypercenter of $G$. The question could then be asked if there is a constant $c >0$ such that if $|Nil_G(x)|/|G|>c$ then $x \in \zeta_\omega (G)$. The following proposition shows that such a constant does not exist.

\begin{prop}
There exists a sequence of groups $(G_k)_{k \in \mathbb{N}}$ together with $x_k \in G_k$ such that $Z(G_k)=1$ and 
\[
\lim_{k \rightarrow \infty} \frac{|Nil_{G_k}(x_k)|}{|G_k|} = 1. 
\]
\end{prop}
\begin{proof}
For $k \in \mathbb{N}$ and $k \geq 2$, let $n=2^k$. Moreover let $\mathbb{K}=\mathbb{F}_{2^n}$ be the field with $2^n$ elements and $V$ be the additive group of $\mathbb{K}$, so that $V$ is an elementary abelian group of size $2^n$.

By Zsigmondy's theorem there exists a prime $p$ which divides $2^n-1$ and doesn't divide $2^l-1$ for any $1 \leq l < n$. Let $P=\langle x \rangle$ be the subgroup of order $p$ in the multiplicative group $\mathbb{K}^\times$. $P$ acts fixed point freely on $V$ by multiplication and the elements of the group $V \rtimes P$ have order either $2$ or $p$.

Let $\mathcal{G}=Gal(\mathbb{K}|\mathbb{F}_{2})$, a cyclic group of order $n=2^k$. Then $\mathcal{G}$ acts both on $V$ and on $P$. If $\sigma \in \mathcal{G}$ is such that $x^\sigma=x$ then $x \in \mathbb{E}=\fix_{\mathbb{K}}(\langle \sigma \rangle)$ the field fixed by $\sigma$. Since $x \notin \mathbb{F}_2$ we have $\mathbb{E}>\mathbb{F}_2$. By the choice of $p$, and since $|x|=p$ has to divide $|\mathbb{E}|-1$, we have that $\mathbb{E}=\mathbb{K}$ so that $\sigma=1$, i.e., $\mathcal{G}$ acts fixed point freely on $P$.

We can consider the group $G=(V \rtimes P)\rtimes \mathcal{G}$, whose order is $2^{n+k}p$.

It is easy to see that there are not any elements of composite order in $G$. In particular $Z(G)=1$. Moreover $N_G(S)=S$ for all $S \in Syl_2(G)$ and if $S_1,S_2 \in Syl_2(G)$, $S_1 \neq S_2$, then $S_1 \cap S_2=V$.
Then $V=O_2(G)$ and so for all $1 \neq v \in V$ 
\[
Nil_G(v)=\bigcup_{S \in Syl_2(G)} S= V \cup (G \setminus (VP)).
\]
Finally 
\[
\frac{|Nil_G(v)|}{|G|}=\frac{2^n+2^{n+k}p-2^np}{2^{n+k}p}= 1-\frac{p-1}{2^{k}p},
\]
which tends to $1$ as $k$ tends to infinity.
\end{proof}

In \cite{lennox:subnormal}, page 238, some candidates for the role of \textit{subnormalizer} are defined, other then the one we used (Definition \ref{subnormalizer}). For example, getting inspiration by the Baer-Suzuki theorem, $S_G^1(H)$ is defined as follows
\[
S_G^1(H) = \lbrace g \in G \ | \ H \sn \langle H,H^g \rangle \rbrace.
\]
As explained in \cite{casolo:subnorsolv} the cardinality of this set can be written as 
\[
|S_G^1(H)|= \delta_G(H)|N_G(H)|
\]
where
\[
\delta_G(H)=\lbrace H^g \ | \ H \sn \langle H,H^g \rangle \rbrace.
\]
The following example shows that there is not an equivalent of Corollary \ref{fitting un terzo} for $S_G^1(x)$, that is a probabilistic version of Baer-Suzuki theorem.

\begin{ese} Let $n=2k$ for $k \in \mathbb{N}, k\geq 2$ and let $G=S_n, x=(1,2)$. We want to count the number of transpositions that generates a $2$-group together with $x$. If $y$ is such a transposition then $y$ commutes with $x$, because otherwise $xy$ would be a $3$-cycle. Then 
\[
\delta_G(x)= \left| \lbrace x \rbrace \cup \lbrace (i,j) \ | \ 2<i<j \leq n \rbrace \right| = 1+\frac{(n-2)(n-3)}{2}
\]
and so
\[
\frac{|S_G^1(x)|}{|G|} = \frac{1+\frac{(n-2)(n-3)}{2}}{\frac{n(n-1)}{2}}
\]
which tends to $1$ as $n$ goes to infinity.
\end{ese}

\subsection*{Acknowledgements} This note is drawn from the author’s PhD thesis which was undertaken at Università degli Studi di Firenze, and written under the patient supervision of Carlo Casolo, to whom a great deal of thanks is owed. The author would also like to thank Francesco Fumagalli and Eugenio Giannelli for their valuable comments and suggestions.

\bibliographystyle{plain}
\bibliography{biblio.bib}

\end{document}